\documentclass{amsart}
\usepackage{latexsym, graphicx, amssymb, cite, hyperref, tikz-cd, tikz}
\usepackage[normalem]{ulem}
\usepackage{subfig}

\theoremstyle{plain}

\newtheorem{corollary}{Corollary}

\newtheorem{definition}{Definition}

\newtheorem{lemma}{Lemma}

\newtheorem{problem}{Problem}
\newtheorem{proposition}{Proposition}
\newtheorem{remark}{Remark}

\newtheorem{theorem}{Theorem}
\numberwithin{equation}{section}
\newcommand{\bth}{\begin{theorem}}
\newcommand{\ble}{\begin{lemma}}
\newcommand{\bcor}{\begin{corr}}

\newcommand{\bdeff}{\begin{deff}}

\newcommand{\bprop}{\begin{proposition}}
\newcommand{\ele}{\end{lemma}}
\newcommand{\ecor}{\end{corr}}
\newcommand{\edeff}{\end{deff}}

\newcommand{\eprop}{\end{proposition}}

	\newcommand{\Rn}{{\mathbb R}^n}

	\renewcommand{\Pi}{\varPi}

	\renewcommand{\epsilon}{\varepsilon}

	\newcommand{\R}{{\mathbb R}}

\begin{document}
		\title[Hearing the shape of a drum]{Hearing the shape of a drum by knocking around}
		
		\author{Xing Wang}

		\address{Department of Mathematics, Hunan University, Changsha, Hunan 410012, PR China}
		\email{xingwang@hnu.edu.cn}
  
		\author{Emmett L. Wyman}
        \address{Department of Mathematics and Statistics, Binghamton University, Vestal NY}
        \email{ewyman@binghamton.edu}

        \author{Yakun Xi}
        \address{School of Mathematical Sciences, Zhejiang University, Hangzhou, Zhejiang 310027, PR China}
        \email{yakunxi@zju.edu.cn}


		\dedicatory{}

		\begin{abstract}
            We study a variation of Kac's question, ``Can one hear the shape of a drum?" if we allow ourselves access to some additional information. In particular, we allow ourselves to ``hear" the local Weyl counting function at each point on the manifold and ask if this is enough to uniquely recover the Riemannian metric. {This is physically equivalent to asking whether one can determine the shape of a drum if one is allowed to knock at any place on the drum.} We show that the answer to this question is ``yes" provided the Laplace-Beltrami spectrum of the drum is simple. We also provide a counterexample illustrating why this hypothesis is necessary.
		\end{abstract}
		
		\maketitle
		
		\section{Introduction}

            In his celebrated paper \cite{Kac66}, Kac asks ``Can one hear the shape of a drum?" In particular, is it possible for a pair of non-isometric Riemannian manifolds (or planar domains) to have identical Laplace-Beltrami spectra?
    
            The right answer to Kac's question depends very much on the context. Some manifolds are spectrally unique, at least among those of a particular class. Amongst planar domains, the disk \cite{Kac66}, and more generally ellipses of small eccentricity \cite{HZ22} are spectrally unique. 
            Moreover, it is known that each planar triangle is spectrally unique amongst the class of planar triangles \cite{durso1990solution, grieser2013hearing}. In a recent paper by Enciso and G{\'o}mez-Serrano \cite{enciso2022spectral}, it is shown that semiregular polygons are spectrally determined amongst the class of convex piecewise smooth planar domains, possibly with straight corners.
            
            However, there is also a rich catalog of negative examples---isospectral pairs of non-isometric manifolds. In an one-page paper, Milnor \cite{Milnor} exhibits such a pair of 16-dimensional tori, answering the general question in the negative. More isospectral pairs of flat tori have been found in as few as four dimensions \cite{schiemann1990beispiel,conway1992four, HeintzeImHof}, and there are provably no such pair of tori in dimensions 2 and 3 \cite{schiemann1990beispiel,schiemann1994ternare}. There are also a large number of isospectral pairs of non-isometric polygons in the plane; see, e.g., \cite{GWW92,BCDS94}.
    
            In \cite{echolocation}, the authors ask a variation on Kac's question: ``Can one hear where a drum is struck?" To describe this problem, we must first define the local Weyl counting function.
	   
		Let $(M, g)$ be a connected compact manifold, consider Laplace-Beltrami eigenfunctions
		$$
		-\Delta_g e_\lambda=\lambda^2 e_\lambda, \quad 0=\lambda_1<\lambda_2 \leq \lambda_3 \leqslant \cdots, \quad\|e_\lambda\|_{L^2(M)}=1 \text {. }
		$$
        { If the boundary $\partial M\neq\emptyset$, we shall assume that the boundary is smooth, and impose the Dirichlet (or Neumann) boundary condition.}
        The Weyl counting function 
            $$
            N(\lambda)=\#\{\lambda_j:\lambda_j\le\lambda\}=\sum_{\lambda_j \leq \lambda} 1,
            $$
        records the Laplacian eigenvalues of $(M,g)$ counting multiplicity.
		Define the local Weyl counting function $N(x, \lambda)$ as
		$$
		N(x, \lambda)=\sum_{\lambda_j \leq \lambda} |e_{\lambda_j}(x)|^2 .
		$$

        In \cite[Section 2]{echolocation}, it is shown that $N(x,{}\cdot{})$ corresponds exactly to the acoustic information generated by a drum when struck at the point $x$. In the same paper, the authors ask if it is possible to determine the location of the point $x$ in $M$, up to symmetry, if we know $N(x,\lambda)$ for all $\lambda$. The answer is ``yes" for some generic class of manifolds. If we do not require $M$ to be connected, there are indeed negative examples, e.g., the disjoint union of Milnor's tori. However, at the time of writing, there are no known negative examples with a single connected component.
            
            The purpose of this paper is to bridge the gap between this question and Kac's. To this end, we ask: ``Can one hear the shape of a drum if one is allowed to knock everywhere on the drum?'' Mathematically, this is:

		\begin{problem}\label{problem 1} 
                Let $M$ be a compact connected manifold equipped with an unknown smooth Riemannian metric $g$. Suppose that the associated local Weyl counting function $N(x,\lambda)$ is known to us for all $x\in M$. Can we determine the metric tensor $g$, and hence the geometry of $M$?
		\end{problem}
        A special case of the above problem is studied in \cite{wang2023surfaces}, where the authors show that if $M$ is a compact boundaryless surface, and the local Weyl counting function $N(x,\lambda)$ is constant in $x\in M$, then $M$ can only be a sphere, a projective sphere, or a flat torus. 
        In this paper, we answer Problem \ref{problem 1} for any manifold with a simple spectrum. Our main result is the following.
            \begin{theorem}\label{main theorem}
                The answer to Problem \ref{problem 1} is ``yes'' for all $(M,g)$ (possibly with a smooth boundary) 
                with a simple spectrum.
            \end{theorem}
            We shall prove Theorem \ref{main theorem} in Section \ref{section: proof} by providing a three step procedure that allows us to recover the metric tensor from $N(x,\lambda)$.
           
            Even though the simplicity condition on the spectrum might seem too restrictive, we shall show in Section \ref{section counterexample} that this condition is necessary in the sense that there exists a pair of non-isometric manifolds with identical local Weyl counting functions.
  
            \begin{proposition}\label{counterexample}
                Fix an integer $d\ge4$. There exists a pair of non-isometric flat tori of dimension $d$, such that their local Weyl functions are identical everywhere.
            \end{proposition}
           Furthermore, Uhlenbeck \cite{Uhlenbeck} proved that for a Baire-generic class of metrics, all Laplace--Beltrami eigenvalues are simple. This fact leads directly to the following corollary of Theorem \ref{main theorem}.

           \begin{corollary}
             Given a compact smooth manifold $M$. There exists a residual class of metrics $\mathcal G_0$ in the $C^\infty$ topology such that the answer to Problem \ref{problem 1} is ``yes'' for any $(M,g),\ g\in\mathcal G_0$.
           \end{corollary}
        \subsection*{Acknowledgements}Y. X. is supported by the National Key Research and Development Program of China No. 2022YFA1007200 and NSF China Grant No. 12171424. X.W. is partially supported by grant No. 531118010864 from Hunan University. E.L.W. is partially supported by NSF grant DMS-2204397.
		
		\section{Preliminaries}
          
		In this section, we recall some definitions and review basic geometric and quantitative properties of eigenfunctions.
		
		\begin{definition}
			Suppose $e_\lambda$ is an eigenfunction, then its nodal set is defined to be its zero set 
			$$
			Z_{e_\lambda} = \{x\in M:e_\lambda(x)=0\}.
			$$
			The nodal domains are defined to be the connected components of $M\backslash Z_{e_\lambda}$.
		\end{definition}
		The famous Courant nodal domain theorem states that when $M$ is a compact surface, the number of nodal domains of the $k\text{th}$ eigenfunction will never exceed $k$. Cheng generalizes this theorem to higher dimensions by studying the regularity of nodal sets (see \cite{N1,N2}). Furthermore, Cheng shows that the nodal set forms a $(n-1)$-dimensional $C^\infty$ submanifold if one excludes a closed set of lower dimensions.
		
		An asymptotic formula about $N(x,\lambda)$ was obtained for paracompact manifolds by H\"omander in \cite{Ho}. As a by-product, one gets the following $L^\infty$ bound about eigenfunctions
		$$
		\|e_\lambda\|_{L^\infty} \le C \lambda^{\frac{n-1}{2}}.
		$$
		Here $C$ is a constant that does not depend on $\lambda$. The sharp $L^p$ bound for all $p\in[2,\infty]$ was obtained by Sogge \cite{So}. The same $L^\infty$ bound was proved by Grieser \cite{grieser2002uniform} for manifolds with boundary.

		\section{Proof of Theorem \ref{main theorem}}\label{section: proof}
In this section, we introduce a three-step procedure that can be used to recover the metric tensor for a compact manifold (possibly with a smooth boundary) with only simple eigenvalues.
		
		\subsection{Step 1. Recovering eigenfunctions} First, we find all eigenvalues $\{\lambda_i\}$ by looking at all jumps of $N(x, \lambda)$ as functions of $\lambda$. Since all eigenvalues are simple, we can find the corresponding absolute values of eigenfunctions $|e_{\lambda_i}(x)|$ for each eigenvalue.
        
        Let  
            $$
                N(x, \lambda_i^{-}) =\lim_{\lambda\to\lambda_i}N(x,\lambda)
            $$
        and
            $$
                E_{\lambda_i}(x) = N(x, \lambda_i)-N(x, \lambda_i^{-}) = |e_{\lambda_i}(x)|^2.
            $$
        We show one can recover $e_{\lambda_i}(x)$ from $E_{\lambda_i}(x)$ with a possible sign change. 
        
        Let $D_1,D_2,\cdots,D_m$ denote the Nodal domains of eigenfunction $e_{\lambda}$. We say two different nodal domains $D_i, D_j$ are adjacent if $D_i\cap D_j$ has non-zero $n-1$ dimensional Hausdorff measure.
		
		Now, we show that $e_\lambda$ has different signs in adjacent nodal domains $D_i, D_j$. This can be done by the method of proof by contradiction. Assume that $e_\lambda(x)$ is non-negative on both $D_i$ and $D_j$. Let $\Sigma=D_i\cap D_j$, then $\Sigma$ is a piecewise smooth submanifold of dimension $n-1$. Choose a smooth piece $\Sigma_0\subset\Sigma$, and let $x_0\in\Sigma_0$ be an interior point. Then we can find an open neighborhood $U$ of $x_0$ on $M$ such that $U\subset D_i\cup D_j$. Let
		$$u(t,x) = e^{\lambda t} e_\lambda(x) \text{ on }[0,1] \times U.$$
		It is easy to check that $u(t,x)$ is harmonic on $[0,1] \times U\textbf{}$, and it has an interior minimal point, which contradicts the maximal principle of a harmonic function.
		
		Now let $\tilde{e}_{\lambda_i}(x)=\delta_j\sqrt{E_{\lambda_i}}(x)$ on $D_j$, where the sign function $\delta_j=\pm 1$ will be chosen by a finite iteration procedure. Let $\delta_1=1$ and $S(0) = \{1\}$, and let 
        \begin{multline*}
        S(l) = \{j:\delta_j \text{ is asigned a value at $l$th iteration cycle}\}
        \subset T_m=\{1,2,\cdots,m\}.  
        \end{multline*}		
		Then, at cycle $l+1$, for each $k\in T_m$ such that there is a $D_k$ is adjacent to $D_j$ for some $j\in S$, we assign $\delta_k=-\delta_j$. 
		
		We claim that for some finite $l>0$, we have $S(l)=T_m$. Indeed, if this is not the case, we will be able to split $ D_j$ into two groups, with unions $U_1$ and $U_2$ respectively, such that $U_1\cap U_2$ is $(n-2)$-dimensional. Then, $U_1\cap U_2$ as a compact set separates the $n$ dimensional connected manifold $M$ into two connected components, so it must have dimension at least $n-1$, a contradiction.
		
		\subsection{Step 2. Recovering the volume form} Endow $M$ with an arbitrary Riemannian metric $\tilde{g}$, then our desired volume form is given by some function $\mu$, such that 
          $$
          dV_g=\mu dV_{\tilde{g}}.
          $$ Since 
		$$\quad \int_M e_{\lambda_i}e_{\lambda_1} d V_g = \delta_{i,1},
		$$
		we then have
			$$\quad \int_M e_{\lambda_i}(x) \tilde{\mu}(x)d V_{\tilde{g}}= 
		\delta_{i,1}.
		$$
        where $\tilde{\mu}(x)=e_{\lambda_1}(x)\mu(x)$
		
		Since $\{e_{\lambda_i}\}$ form a basis of $L^2(M, dV_g)$,  they are also a basis of $L^2(M, dV_{\tilde{g}})$ thanks to the fact that the measures are comparable $dV_g \sim dV_{\tilde{g}}$.
		
		Applying the Gram-Schmidt process with respect to $dV_{\tilde{g}}$, we can find coefficients $\{a_{kj}\}$ such that
		$$
		\phi_k=\sum_{j=1}^k a_{k j} e_{\lambda_j}, \qquad k=1,2, \ldots
		$$
		form an orthonormal basis of $L^2(M, dV_{\tilde{g}})$.
		
		Now 
		$$\langle\phi_k, \tilde{\mu}\rangle_{\tilde{g}}=\sum_{j=1}^k a_{k j}\langle e_{\lambda_j}, \tilde{\mu}\rangle_{\tilde{g}}=a_{k 1} $$
		which gives us
		$$
		\quad \tilde{\mu}=\sum_{k=1}^{\infty} a_{k 1} \phi_k(x).
		$$
        Then
        $$
		\mu(x)=[e_{\lambda_1}(x)]^{-1}\tilde{\mu} \text { for }x \notin Z_{e_{\lambda_1}}.
		$$
        Noting that $Z_{e_{\lambda_1}}$ is at most $(n-1)$-dimensional {(and usually empty by Courant's nodal domain theorem)}, the equation above uniquely determines $\mu$ as a continuous function on $M$. Therefore, by $dV_g=\mu dV_{\tilde{g}}$, we recover the volume form $dV_g$.
		
		\subsection{Step 3. Recovering the metric} Now we have recovered $\{e_{\lambda_i}\}$ and $dV_g$. By the spectral theorem, for $f\in C^\infty (M)$, we get
		$$
		-\Delta_g f= \lim_{\lambda\to\infty}\sum_{\lambda_j\leq \lambda} \lambda_j^2 \int_M e_{\lambda_j}(x) e_{\lambda_j}(y) f(y) dV_g.
		$$
		The limit is always well defined since the Fourier coefficients of $f$ with respect to $\{e_{\lambda_i}\}$ decay rapidly and $\|e_{\lambda}\|_{\infty}\le C \lambda^{\frac{n-1}{2}}$. For any interior point $p\in M$, we choose an open neighborhood $U$ contained in the interior of $M$. Let $x_0$ be the local coordinates of $p$. Take a smooth bump function $\phi\in C_c^\infty(U)$ such that $\phi(x_0)=1$. For any $v\in \Rn$, let $f_v(x;\lambda)=e^{\lambda v\cdot (x-x_0)}\phi(x)$. Then
		$$
		\lim_{\lambda\to\infty} -\lambda^{-2}\Delta_g f_v(x_0;\lambda)= g^{ij}(x_0)v_iv_j.
		$$
		gives the metric tensor in the interior of $M$, which extends uniquely to the boundary by continuity. 
		
		\begin{remark}
            Indeed, we do not need to know whether the spectrum of the manifold is simple or not in advance. Basically we can execute the above procedure for any given smooth manifold to produce a metric on it, and then check whether the resulting Riemannian manifold give the desired local Weyl counting function or not.
        \end{remark}
		
		\section{Proof of Proposition \ref{counterexample}}\label{section counterexample}
  In this section, we show that if we drop the assumption that the manifold has a  simple spectrum, then there exist pairs of manifolds that one cannot distinguish by knocking around. Milnor found two 16-dimensional flat tori which are isospectral but not isometric. Since then, isospectral pairs of flat tori have been found in any dimensions greater or equal to four \cite{schiemann1990beispiel,conway1992four, HeintzeImHof}. Furthermore, Schiemann showed that there is no such pair of tori in dimensions three and lower \cite{schiemann1990beispiel,schiemann1994ternare}. See the survey article \cite{nilsson2023isospectral} for a detailed discussion. 

Now we are ready to prove Proposition \ref{counterexample}. Take two non-isometric, isospectral flat tori of dimension $d\ge4$, say $\mathbb T$ and $\mathbb T'$, with Weyl counting functions $N_{\mathbb T}(\lambda)$ and $N_{\mathbb T'}(\lambda)$ respectively.

Since they are isospectral, their Weyl counting functions are identical 
\[N_{\mathbb T}(\lambda)=N_{\mathbb T'}(\lambda),\quad \forall \lambda\in\mathbb R.\]
On the other hand, since $\mathbb T$ is a flat torus,
its isometry group acts transitively on points in $\mathbb T$. The local Weyl counting function then satisfies
\[N_{\mathbb T}(x,\lambda)=f(\lambda),\quad \forall x\in\mathbb T.\]
Similarly, 
\[N_{\mathbb T'}(x,\lambda)=g(\lambda),\quad \forall x\in\mathbb T'.\]
Note that 
\[{\rm Vol}(\mathbb T)f(\lambda) = \int_\mathbb TN_{\mathbb T}(x,\lambda)dx=N_{\mathbb T}(\lambda)=N_{\mathbb T'}(\lambda)=\int_{\mathbb T'}N_{\mathbb T'}(x,\lambda)dx={\rm Vol}(\mathbb T')g(\lambda).
\]
Noting that by the Weyl law, ${\rm Vol}(\mathbb T')={\rm Vol}(\mathbb T)$, we see that $N_{\mathbb T}(x,\lambda)=f(\lambda)$ and $N_{\mathbb T'}(x,\lambda)=g(\lambda)$ must be the same function of $\lambda$. That is,
\[N_{\mathbb T}(x,\lambda)=N_{\mathbb T'}(y,\lambda)=f(\lambda),\quad \forall x\in\mathbb T, \ \forall y\in\mathbb T'.\]

This is an example of a pair of non-isometric manifolds with the same topology and differential structure which cannot be distinguished by looking at their local Weyl counting functions alone. Incidentally, their spectra are not simple.
		
	\section{Obstructions for the General Case in Dimensions Two and Three}

        Note that the counterexamples to Problem \ref{problem 1} discussed in the previous section only appear in dimensions four or higher. An interesting question is whether one can fully solve Problem \ref{problem 1} in dimensions two or three. This question appears intriguing yet difficult. In many cases, when the spectrum is not simple, it will be impossible to uniquely determine the local expression of the metric. We shall demonstrate this fact with the following example.
        
        \begin{proposition}
        There exists a smooth manifold $M$ which admits two distinct metrics $g_1$ and $g_2$ such that $N(x,\lambda)$ is identical at each point, but $(M,g_1)$ and $(M,g_2)$ are nevertheless isometric.
        \end{proposition}
        
        \begin{proof}
        We prove the proposition by providing an explicit simple example.

            Let $M = {\mathbb T}^2 = \R^2/{\mathbb Z}^2$. Take
            \[
                g_1 = \begin{bmatrix}
                    1 & 0 \\
                    0 & 1
                \end{bmatrix}
                \qquad \text{ and } \qquad
                g_2 = \begin{bmatrix}
                    1 & 1 \\
                    1 & 2
                \end{bmatrix}.
            \]
            The mapping ${\mathbb T}^2 \to {\mathbb T}^2$ induced by the linear shear
            \[
                Tx = \begin{bmatrix}
                    1 & 1 \\
                    0 & 1
                \end{bmatrix} x
            \]
            is an isometry $({\mathbb T}^2, g_2) \to ({\mathbb T}^2, g_1)$.
            \end{proof}

    	We remark that our proof of Theorem \ref{main theorem} indicates that the above phenomenon cannot occur when the spectrum is simple.
    
        For manifolds with boundary, we can construct similar counterexamples. For instance, let $D$ be the unit disk in the plane. In polar coordinates $(r,\theta)$, we can construct a vector field $X = \phi(r) \partial_\theta$, where $\phi \in C^\infty(\mathbb{R})$ satisfies $\phi(r) \equiv 0$ for $r \leq \frac{1}{4}$ and $\phi(r) \equiv 1$ for $r \geq \frac{1}{2}$.
        
        Let $\Phi_t$ be the flow generated by $X$, and let $g_0$ be the standard Euclidean metric on $D$. Then for $g_1 = \Phi_t^* g_0$, one can see that $(D, g_0)$ and $(D, g_1)$ are isometric and share the same $N(x, \lambda)$ since it is radial, while $g_0$ and $g_1$ are quite different in local coordinates.

\bibliography{reference}{}
\bibliographystyle{alpha}

	\end{document}